\documentclass[11pt]{amsart}

\usepackage{amsmath}
\usepackage{amssymb}
\usepackage{graphicx}

\usepackage{tasks}
\usepackage{hyperref}
\usepackage{blindtext}
\usepackage{setspace}
\usepackage{marginnote}
\usepackage{amsmath,amssymb}
\usepackage{color}
\usepackage{amsrefs}
\usepackage{mathtools, microtype}
\usepackage{graphicx}  
\usepackage{dirtytalk}
\usepackage[top=3.5cm, bottom=2.5cm, outer=2.5cm, inner=2.5cm]{geometry}% margines

\usepackage{hyperref}
\hypersetup{linkcolor=red,
citecolor=black
}
\usepackage{amsmath}

\newenvironment{Ali}{\color{blue}}{\color{black}}
\newcommand{\AAA}{\begin{Ali}}
\newcommand{\PPP}{\end{Ali}}

\newcommand{\RRR}{\begin{Ali}}
\newcommand{\SSS}{\end{Ali}}

\newcommand{\Rey}{\mathrm{Re}_{\nu}}
\newcommand{\rey}{\mathrm{Re}_{\bar{\nu}}}

\newcommand{\E}{\mathbb{E}}
\newcommand{\grad}{\nabla}

\allowdisplaybreaks[4]

\newtheorem{theorem}{Theorem}[section]

\newtheorem{lemma}[theorem]{Lemma}

\theoremstyle{definition}
\newtheorem{definition}[theorem]{Definition}

\newtheorem{remark}[theorem]{Remark}
\newtheorem{AS}[theorem]{Assumption}

\numberwithin{equation}{section}
\usepackage[utf8]{inputenc}
\usepackage[english]{babel}
\usepackage{paralist}
\usepackage{amsthm}

\usepackage{bm,amsmath}

\newcommand{\RN}[1]{%
  \textup{\uppercase\expandafter{\romannumeral#1}}%
}

\allowdisplaybreaks[4]

\date{\today}
\begin{document}

\title[Statistics of the 3D Stochastic Turbulence Model]{Time-averaged statistics of the  3D stochastic Ladyzhenskaya-Smagorinsky equations}

\author[L. W.-T. Fan]{Louis Wai-Tong Fan}

\address[L. W.-T. Fan]{School of Data Science and Society,
University of North Carolina at Chapel Hill,
NC, USA}
\email{\tt louisfan@unc.edu}

\address[L. W.-T. Fan]{Department of Organismic and Evolutionary Biology,
Harvard University,
MA, USA}
\email{\tt wlfan@fas.harvard.edu}

\author[A. Pakzad]{Ali Pakzad}

\address[A. Pakzad]{Department of Mathematics, California State University Northridge,  CA, USA}
\email{{\tt pakzad@csun.edu}}

\maketitle
\setcounter{tocdepth}{1}

\begin{abstract}
 
Due to the chaotic nature of turbulence, statistical quantities are often more informative than pointwise characterizations. In this work, we consider the stochastic Ladyzhenskaya–Smagorinsky equation driven by space-time Gaussian noise in a three-dimensional periodic domain.  We derive a rigorous upper bound on the first moment of the energy dissipation rate and show that it remains finite in the vanishing viscosity limit, consistent with Kolmogorov's phenomenological theory. This estimate also agrees with classical results obtained for the Navier–Stokes equations and demonstrates that, in the absence of boundary layers, as considered here, the model does not over-dissipate.
\end{abstract}

\section{Introduction}

In the late 1960s, Ladyzhenskaya \cite{L67, L69, L98, L68} and Smagorinsky \cite{SM63} proposed several modifications to the Navier–Stokes equations, with the  viscosity depending on the gradient of the velocity field. These modified equations form a significant mathematical model for capturing the flow behavior of various non-Newtonian fluids and are also widely used in simulating the large eddies of turbulent flows \cite[Chapter~8]{Wilcox06}. Herein, we consider the stochastically-forced Ladyzhenskaya–Smagorinsky equations
\begin{equation}\label{SSM}
\begin{split}
 d u +  \left( u \cdot \nabla u -  \nu \Delta u - \nabla \cdot \left( \bar{\nu} \,  |\nabla u|^{r-2} \nabla u \right) + \nabla  p \right) \, dt & =    f \,  dt  +   g(t,u) \, d W,\\
 \nabla \cdot u & = 0,
\end{split}
\end{equation}
in a  periodic box of side length $\ell$, i.e. with $ x= (x_1, x_2, x_3) \in \mathbb{T}^3$, the $3-$dimensional torus.   In the equation above,  the stochastic process $u$ and $p$ are the velocity field and the pressure respectively,  $\nu$ is the  kinematic viscosity,  $\bar{\nu} >0$ is the model parameter and Frobenius norm   $|\nabla u|^2 = \nabla u : \nabla u$.   The two terms on the right-hand side of \eqref{SSM} are, respectively, a deterministic  time-independent applied force $f$   and a stochastic forcing $g$ driven by a space-time Gaussian noise $W$. The multiplicative noise term \( g(t, u)\, dW \) has a continuous-in-time effect on the system, making it well-suited for stochastic models of fluid dynamics, where it can represent random interactions occurring over time. Recently, Nguyen, Tawri and Temam \cite{NTT21} established the existence and uniqueness of both martingale and pathwise solutions for a class of stochastic Ladyzhenskaya–Smagorinsky equations driven by Gaussian noise and a Poisson random measure. 
Their analysis, based on the Galerkin approximation method, can be straightforwardly adapted to yield the existence of a martingale solution for the stochastic equation considered in this paper.

A salient characteristic of turbulence is its intermittency in both space and time, which makes it impractical to provide a detailed pointwise description of the fluid flow; see numerous experimental examples in \cite{baumert2005}. However, averaged quantities often exhibit predictable behavior, even when the flow dynamics are highly irregular over finite time intervals. As a result, theoretical studies of turbulence typically adopt a statistical rather than deterministic, pointwise approach, with turbulence most commonly quantified through the turbulent energy dissipation rate \cite{F95}. 
In this article, we derive rigorous upper bounds on the bulk energy dissipation rate \(\varepsilon\) for incompressible fluid flows governed by \eqref{SSM}, and quantify the effect of the noise on this key measure of turbulent flow.

\subsection{Turbulent Energy Cascade}

A fundamental principle of modern hydrodynamic turbulence theory is that nonlinear interactions between components of the velocity field transfer energy from the directly forced large spatial scales, through the so-called inertial range, down to small dissipative scales where viscosity effectively consumes kinetic energy and converts it into heat. This multiscale transfer of energy, from large scales down to small dissipative scales, is known as the energy cascade, a concept originally proposed by Richardson in 1992.  He illustrated the idea of energy transfer in his book  \cite{richardson1922} as:  \say{\textit{Big whirls have little whirls that feed on their velocity; 
And little whirls have lesser whirls, and so on to viscosity}.}

A fundamental and surprising consequence of the energy cascade mechanism is the so-called \emph{anomalous dissipation}, which refers to the persistence of a finite, non-zero energy dissipation rate even as the viscosity \(\nu\) approaches zero (equivalently, in the limit of infinite Reynolds number). Intuitively, one might expect that vanishing viscosity would lead to zero dissipation; however, due to the formation of increasingly fine-scale velocity fluctuations in turbulent flows, energy continues to be dissipated at small scales. This phenomenon is quantitatively captured by Kolmogorov’s scaling law  \cite{K41} for the energy dissipation rate
\[
\varepsilon \sim C_{\epsilon} \frac{U^3}{L} \quad \text{as } \quad \nu \to 0,
\]
where \(\varepsilon\) denotes the total energy dissipation rate per unit mass, \(U\) is  a characteristic velocity scale, \(L\) is an integral length scale in the flow characterizing the domain or a large scale in the forcing and flow, and \(C_{\epsilon}\) is a dimensionless constant. This theory yields many physical predictions about averaged quantities in turbulent flows that are remarkably accurate, yet remain mathematically unproven. Nevertheless, significant efforts have been devoted to rigorously \textit{estimating} this key physical quantity using mathematical analysis.  

The idea of extracting information about driven turbulent flows by deriving bounds on physically relevant quantities, using mathematically rigorous methods and without relying on ad hoc assumptions, originates in the work of Howard \cite{H72} and Busse \cite{B70}. Building on these foundations, Doering and Constantin \cite{DC92}, and later Doering and Foias \cite{DF02}, developed a practical framework for rigorously estimating the energy dissipation rate directly from the equations of motion. They showed that the dissipation rate  satisfies an upper bound of the form
\[
\varepsilon \lesssim  \frac{U^3}{L} \ \quad \text{as } \quad \nu \to 0.
\]
These foundational results have inspired numerous developments in various directions, see for example \cite{W00, PS2024, AP17, L16, K16, DLPRSZ18, AP19, BJMT14, FGHR08, W97, kumar2025}.  Sreenivasan has compiled extensive experimental \cite{sreenivasan1984} and numerical \cite{sreenivasan1998} evidence supporting Kolmogorov scaling. 

\subsection{Randomness in Turbulence}
Stochastic fluid dynamics emerged as a natural extension of classical deterministic models to account for the influence of unresolved small-scale fluctuations, environmental variability, and modeling  both numerical and empirical uncertainties in fluid flows (see Chapter 3 of \cite{P00}).  In addition, there are many examples which support the stabilization of PDEs by noise (see, e.g. \cite{B01},  \cite{CLM01},  \cite{FSQD19}, \cite{K99}).    Early developments in the 1960s and 70s were driven by efforts to incorporate random perturbations into the Navier–Stokes equations to better capture observed turbulence features and to provide more realistic representations of atmospheric and oceanic flows.  Seminal works by Bensoussan and Temam \cite{BT73} laid foundational theory for stochastic partial differential equations (SPDEs) in fluid mechanics, introducing rigorous formulations of noise-driven fluid systems. Since then, a rich body of literature has developed around the stochastic Navier–Stokes equations, and other related models, with significant advances in existence theory, invariant measures, ergodicity, and long-time statistical properties (see for example  \cite{BCPW19},  \cite{B00},  \cite{BP00}, \cite{CI08}, \cite{CI11}, \cite{DGTZ12}, \cite{GKVZ14}, \cite{KUZ18},\cite{MR04}, \cite{MR05}, and  \cite{WW15} and the references therein). Building on this framework, the stochastic Ladyzhenskaya equations \eqref{SSM} arose as an important generalization to model non-Newtonian fluids and turbulence in large-eddy simulations, incorporating nonlinear viscosity terms inspired by Ladyzhenskaya’s and Smagorinsky’s modifications. The model corresponding to  $r=2$  reduces to the stochastic Navier-Stokes equations,  and   for $r=3$,   it is mathematically equivalent to  the Stochastic Smagorinsky model \cite{SM63} and the stochastic  NSE with the von Neumann Richtmyer artificial viscosity for shocks \cite{VR50}. Recent studies have extended well-posedness results to this setting \cite{cyr2020review, NTT21}, laying the groundwork for deeper investigations into the statistical properties  of stochastic turbulence under nonlinear viscosity models. Much of the existing literature addresses both pathwise and martingale solutions. However, in this manuscript, we focus exclusively on the martingale solutions of \eqref{SSM}, which are probabilistically weak analogues of the Leray-Hopf weak solutions in the deterministic setting.

In the context of stochastic turbulence, Kolmogorov’s \(4/3\)- and \(4/5\)-laws, as well as the third-order universal scaling law, have been rigorously studied in \cite{BCPW19} by Bedrossian et al.\ and in \cite{Dudley2024} by Dudley, respectively.  In the context of shear flow, the authors in \cite{FJP21} and \cite{FPTT23} derived upper bounds on the first and second moments of the energy dissipation rate in shear-driven turbulence, where the flow is driven only by random boundary motion. Biferale et al.\ \cite{biferale2004, biferale2004anomalous} performed numerical simulations of the three-dimensional Navier–Stokes equations driven by a stochastic body force that is white in time. Alexakis and Doering in \cite{alexakisdoering2006} developed a systematic approach to estimate energy and enstrophy dissipation in 2D forced flows, and this work has been recently extended to two-dimensional flows driven by various stochastic forces in \cite{kumar2025}.

%The above  model corresponding to $r = 2$ reduces to the stochastic  Navier–Stokes equations and also  exhibits equivalence to well-known turbulence models, including large eddy simulation and zero-equation models, for specific values of $r$, such as $r=3, 4$.  In both practical applications and turbulence modeling, the paramount importance lies in the behavior of averaged quantities, which are frequently simulated and extensively studied. 

\subsection{Summary of the results}

In this manuscript, considering turbulence driven by the stochastic model   \eqref{SSM}, we rigorously prove (see Theorem \ref{MainThm1}) that the first moment of the energy dissipation rate $\varepsilon$ is estimated as 

\[
\varepsilon \lesssim  \left(1 + \frac{\rho_{\infty}\, L}{U}+ \frac{\nu}{U\, L} +  \frac{\bar{\nu}\, U^{r-3}}{L^{r-1}} \right) \frac{U^3}{L}.
\]
The above expression is dimensionally consistent, and \( \rho_{\infty} \in[0,\infty) \) is a constant that quantifies the contribution of the stochastic forcing; see Assumption \ref{Assumption}. Note that the above estimate 

\begin{itemize}
    \item indicates that the energy dissipation rate  \( \varepsilon \) balances  the energy input rate \( U^3/L \), in agreement with classical turbulence phenomenology (e.g., \cite{F95}), 
    \item  remains finite in the high Reynolds number limit reflects the so-called \emph{dissipative anomaly} and is consistent with Kolmogorov's theory of turbulence \cite{K41}, 
    \item is consistent with the rate observed for the the Navier–Stokes equations (when \( g \sim 0 \) and \( \bar{\nu} \sim 0 \)) in \cite{DF02},
     \item is consistent with the rate observed for the Smagorinsky model of turbulence (when \( g \sim 0 \) and $r=3$) in \cite{L16},
    \item establishes that, in the absence of  boundary layers, the  model does not over dissipate.
\end{itemize}

\subsection*{Organization of this paper}
In section \ref{Section2},  we introduce the precise assumptions on both the deterministic and stochastic forces, and also review the necessary setup and results required for this paper. In Section \ref{Section3}, we state the main results.

%This paper is concerned with the three-dimensional Navier-Stokes equation for an incompressible fluid in the absence of boundaries, in the long time limit, with a random perturbation of the body forces, (\ref{SNSE}). 

 \section{Mathematical framework and definitions}\label{Section2}

 Let $D$ denote the   periodic box in $3d$ with side length $\ell$, i.e. with $ x= (x_1, x_2, x_3) \in \mathbb{T}^3$, the $3-$dimensional torus of volume $ |D|= \ell^3$ (The analysis holds for both $2$ and $3$ dimension, but we will focus on the applications in  $ 3d$).  Constant $C$ represents a generic positive constant independent of $\nu, \Rey, $ and other model parameters.

 %{\color{red}("bounded open domain" is not correct, as we need a BOX with periodic boundary condition. Perhaps use the notation $\mathbb{T}^2$, the torus as in Bedrossian' etal's paper?)}

Following \cite{BT73} and \cite{RocknerZhuZhu2015}, we shall restrict ourselves to flows which have zero average over the domain, and with an eye on the   notations  in \cite[Example 8.1]{NTT21},  we define  
\begin{equation}\label{Def:VH}
\begin{split}
H: &=\{\varphi\in \left(L^2(D)\right)^3:\,\nabla\cdot \varphi=0,  \, \,  \int_D \varphi=0\}, \\
V: & =\{ \varphi\in \left(H^1(D)\right)^3:\,\nabla\cdot \varphi=0, \, \, \int_D \varphi=0\},
\end{split}
\end{equation}
which are  the completions of divergence-free smooth functions with zero average with respect to $L^2$ and $H^1$-norms.  We also let 
\begin{equation}\label{Def:Vr}
\begin{split}
V_r: & =\{ \varphi\in \left(W^{1,r}(D)\right)^3:\,\nabla\cdot \varphi=0,  \, \, \int_D \varphi=0\}.
\end{split}
\end{equation}

 Let $p \in [1 , \infty]$, and the
Lebesgue space $L^p(\Omega)$ is the space of all measurable functions $v$ on $\Omega$ for which
\begin{align*}
\|v\|_{L^p} & :=\big( \int_{\Omega} |v (x)| ^p\, dx\big)^{\frac{1}{p}} \, < \infty,  \hspace{1cm} \text{if} \hspace{0.2cm} p \in [1 , \infty),\\
\|v\|_{L^{\infty}}& := \sup_{x \in \Omega}  |v (x)| < \infty,  \hspace{2.3cm} \text{if} \hspace{0.2cm} p  = \infty.
    \end{align*}
The $L^2$  norm and inner product will be denoted by $\|\cdot\|$ and $( \cdot ,  \cdot)$ respectively, while all other norms will be labeled with subscripts. Let $B$ be a Banach space %of functions defined on $\Omega$ 
with norm $\| \cdot \|_B$. We denote by $L^p(a, b; B)$,  $p \in [1 , \infty]$,  the space of functions $v : (a, b) \rightarrow  B$ such that
\begin{align*}
\|v\|_{L^p (a, b ; B)} &:=\big( \int_a^b \|v(t)\|_B ^p\, dt\big)^{\frac{1}{p}} \, < \infty,  \hspace{1cm} \text{if} \hspace{0.2cm} p \in [1 , \infty),\\
\|v\|_{L^{\infty} (a, b ; B)}& := \sup_{t  \in (a , b)} \|v( t )\|_B < \infty,  \hspace{1.87cm} \text{if} \hspace{0.2cm} p  = \infty.
    \end{align*}

Denote by \( p' \) the conjugate exponent, where \( \frac{1}{p} + \frac{1}{p'} = 1 \). Assume \( \phi \in L^p \) and \( \psi \in L^{p'} \) with \( 1 \leq p \leq \infty \). Then, the following holds
\begin{equation}\label{Holder}
 \tag{H\"older's inequality}
\|\phi \, \psi\|_{L^1} \leq \|\phi\|_{L^p} \, \|\psi\|_{L^{p'}}.
\end{equation}
Additionally, for any \( a, b \geq 0 \) and \( \lambda > 0 \), we have
\begin{equation}\label{Young}
 \tag{Young's inequality}
a b \leq \lambda a^p +  \left( p \lambda \right)^{-\frac{p'}{p}} \frac{1}{p'} b^{p'}.
\end{equation}
Writing $\lambda_1$ as the smallest eigenvalue of  the Stokes operator (see \cite{FMRT01}),  then for $\phi \in H^1$,  with the zero average assumption,  we have 
\begin{equation}\label{Poincare}
\tag{Poincar\'e inequality}
\lambda_1 \|\phi\|^2 \leq  \|\nabla \phi\|^2.
\end{equation}

\subsection{Martingale solution} 

The concept of martingale solutions for stochastic Navier-Stokes equations, as described by Flandoli and G\polhk atarek in \cite{FG95}, is analogous to those for stochastic ordinary differential equations. In this notion of solution, the underlying filtered probability space is part of the solution, rather than an input of the system. A bit more precisely, the \textit{pair}--consisting of the filtered probability space and the process defined on it that satisfies the equation--is a martingale solution. 

We will precise such notion for \eqref{SSM} in Definition \ref{MgaleCmpct}, after introducing some notations below.  Consider the  Gaussian noise $W$ which  is assumed to be a mean zero, white-in-time and colored-in-space Gaussian process that is a $S$-valued $Q$-Wiener process (c.f. \cite{DZ14}). Here,  $S$ is a separable Hilbert space and  $Q$ is a trace class nonnegative operator on $S$. More explicitly, we let $S_0:= Q^{1/2}(S)$, and $\{e_k\}$ be a complete orthonormal system in $S_0$. Then
(c.f. Proposition 4.3 of \cite{DZ14}) $W$ can be represented by $$W(t,x)=\sum_{k=1}^{\infty}\sigma_k\,e_k(x)\,B_k(t),$$
where  $\{\sigma_k^2\}$ is a  bounded sequence of nonnegative numbers that are the corresponding eigenvalues (i.e. $Qe_k=\sigma_k^2 e_k$), and  $\{B_k\}$ are independent one-dimensional Brownian motions. We assume the  coloring condition
\begin{equation}\label{E:EW}
  E_{W}:=\frac{1}{2}  \sum_{k=1}^{\infty}\sigma_k^2<\infty.
\end{equation}
We assume that $W$ is defined on a filtered probability space $\big(\Omega,\,\mathcal{F},\, (\mathcal{F}_t)_{t\in[0,\infty)}, \,\mathbb{P}\big)$ that satisfies the usual conditions.

We  let  $L_2(S_0,H)$ be the space of Hilbert-Schmidt operators from $S_0$ to $H$ equipped with Hilbert-Schmidt norm 
$$\|\Phi\|^2_{L_2(S_0,H)}:= \sum_{k=0}^{\infty}|\Phi e_k|^2, \hspace{1cm} \forall\,  \Phi\in L_2(S_0,H).$$ 

Martingale solutions were first constructed in \cite{BT73} for stochastically forced Navier–Stokes equations  and later on  in \cite{NTT21} for   the stochastically-forced Ladyzhenskaya–Smagorinsky equations  \eqref{SSM}. 
%The definition and the existence of a martingale solution to \eqref{SSM} follow from Definition 3.1 and Theorem 3.3  of \cite{NTT21} respectively.
\begin{definition}
\label{MgaleCmpct}
A  martingale solution to \eqref{SSM}
on $[0, T ]$ consists of a stochastic basis $\big(\Omega,\, (\mathcal{F}_t)_{t\in[0,T]}, \,\mathbb{P}, W\big)$ 
%that satisfies the usual conditions,
with a complete right-continuous filtration $(\mathcal{F}_t)_{t\in[0,T]}$, 
and an $\mathcal{F}_t$-progressively measurable stochastic process
\[
u:\; [0,T] \times \Omega  \rightarrow  H
\]
such that 
\begin{itemize}
\item $u\in L^2(\Omega, L^{\infty}((0,T);H))\cap L^2([0,T]\times \Omega;V)
\cap L^r([0,T]\times \Omega;\,V_r )$
\item $u(0)\in H$ is $\mathcal{F}_0$-measurable 
\item for all $t\in [0,T]$, the following equation holds $\mathbb{P}$-almost surely, 
\begin{equation}\label{SSM2}
\begin{split}
 & u(t)-u(0) + \int_0^t \Big( u \cdot \nabla u -  \nu \Delta u - \nabla \cdot ( \bar{\nu} \,  |\nabla u|^{r-2} \nabla u ) + \nabla  p \,-f\Big)\, ds \\
  & = \int_0^t  g(s,u(s)) \, d W (s).
\end{split}
\end{equation}
\end{itemize}
\end{definition}

\smallskip

The following existence theorem follows directly from  \cite[Theorem~3.3 and Example~8.1]{NTT21}.

%The following existence theorem follows easily from \cite[Theorem 3.3 and Example 8.1]{NTT21} and \cite{MR05} despite the fact that here we consider the periodic boundary condition instead of the Dirichlet boundary condition.

\begin{theorem}

Suppose $u_0\in H$ and the functions $f:D \to \mathbb{R}_+$ and $g: \mathbb{R}_+\times H\to L_2(S_0,H)$ satisfy that
\begin{itemize}
\item $u(0)$  is  $\mathcal{F}_0$-measurable,
    \item $f\in V$, $\grad f \in L^{\infty}(D)\cap L^r(D)$, 
    \item for any $T\in(0,\infty)$, there exists a constants $\rho_{T},\,\widetilde{\rho}_{g,T}\in(0,\infty)$ such that
\begin{equation}\label{Cond1:g}
\sup_{t\in [0,T]}\|g(t,v)\|^2_{L^2(S_0,H)} \leq \rho_{T}\, (1+|v|_H^2) \quad \text{for all }v\in H, 
\end{equation}
%(if $g(t,v)=g(t)$ does not depend on $v$, then this reads as $\sup_{t\in [0,T]}\|g(t)\|^2_{L^2(S_0,H)} \leq \rho_{T}$) 
and that
\begin{equation}\label{Cond2:g}
\sup_{t\in [0,T]}\|g(t,v)-g(t,w)\|^2_{L^2(S_0,H)} \leq \widetilde{\rho}_{g,T}\,|v-w|_H^2 \quad \text{for all }v,w\in H, 
\end{equation}
\end{itemize}
Then for any $r\in (2,\infty)$ and $T\in (0,\infty)$, there exists  a martingale solution to \eqref{SSM} on $[0,T]$ in the sense of Definition \ref{MgaleCmpct}. Furthermore, there exists a constant $C_T\in(0,\infty)$  such that 
\begin{equation}\label{EnergyBounds}
    \E\left[ \sup_{t \in [0,T]} \|u(t)\|^2\right] + \nu \, \E\left[ \int_0^T \|\nabla u\|^2\, dt\right] + \bar{\nu} \, \E\left[ \int_0^T \|\nabla u\|_{L^r}^r\, dt\right] \leq C_T \,\left(1+\E\left[\|u_0\|^2\right]\right). 
\end{equation}
\end{theorem}

Suppose that the pair $\left(\big(\Omega,\, (\mathcal{F}_t)_{t\in[0,T]}, \,\mathbb{P}, W\big),\;u\right)$ is a martingale solution.
In particular, the stochastic integral 
$$\int_0^t g(s,u(s))\,dW(s)=\sum_{k=1}^{\infty}\sigma_k\int_0^t g(s,u(s))\,e_k\,dB_k(s)$$ is a well-defined $(\mathcal{F}_t)_{t\in[0,\infty)}$-martingale with mean zero and variance
\begin{align}
&\mathbb{E}\left[\left|\int_0^t g(s,u(s))\,dW(s)\right|_{H}^2\right]
=\left(\sum_{k=1}^{\infty}\sigma_k^2\right)\mathbb{E}\int_0^t \|g(s,u(s))\|^2\,ds \notag\\
\leq & \left(\sum_{k=1}^{\infty}\sigma_k^2\right) \rho_{T}\,\mathbb{E}\int_0^t (1+|u(s)|_{H}^2)\,ds
<\infty,  \label{Var_g}
\end{align}
for $t\in[0,T]$, where the inequality in \eqref{Var_g} follows from  the growth condition \eqref{Cond1:g}.
Suppose we write
\[
g(t,u(t))\, dW(t) = \sum_{k=1}^{\infty} g_k(t,u(t))\,dB_k(t),
\]
where, as in \cite{MS02}, $g_k(t,u(t)):=\sigma_k\,g(t,u(t))\,e_k\in H$. Then the left-hand side of \eqref{Var_g} is equal to $\int_0^t\sum_k \sigma_k^2\, |g_k(s,u(s))|^2_H\,ds$.

More generally, suppose $\phi=(\phi(t))_{t\geq 0}$ is an $(\mathcal{F}_t)_{t\in[0,\infty)}$-adapted, c\'adl\'ag, stochastic process with values in $L_2(S_0,H)$ and such that $\mathbb{E}\int_0^t \|\phi(s)\|^2\,ds <\infty$.
Then the stochastic integral 
$$\int_0^t \phi(s)\,dW(s)=\sum_{k=1}^{\infty}\sigma_k\int_0^t \phi(s)e_k\,dB_k(s)$$ is a well-defined $(\mathcal{F}_t)_{t\in[0,\infty)}$-martingale with mean zero and variance $$\mathbb{E}\left[\left|\int_0^t \phi(s)\,dW(s)\right|_{H}^2\right]
=\left(\sum_{k=1}^{\infty}\sigma_k^2\right)\mathbb{E}\int_0^t \|\phi(s)\|^2\,ds <\infty.$$
Recalling the trace $\text{Tr}(\Phi)=\sum_{k}(\Phi e_k, e_k)$ for any $\Phi\in L_2(S_0,H)$, we observe that 
$$\text{Tr} (g^*g)=\sum_k(g^* g e_k, e_k)=\sum_k |ge_k|_H^2=\|g\|^2.$$

\medskip

We make the following assumption on the time-inhomogeneous $g$ throughout this paper.
\begin{AS}\label{Assumption}
Let $\rho_{\infty}$ be the smallest constant such that

\begin{equation}\label{Cond1:g_infty}
\sup_{t\in [0,\infty)}\|g(t,v)\|^2_{L_2(S_0,H)} \leq \rho_{\infty}\, (1+|v|_H^2) \quad \text{for all }v\in H.
\end{equation}

We assume that $\rho_{\infty}<\infty$, i.e. we assume it exists and is finite.
\end{AS}

\begin{remark}
If $g(t,v)=g(t)$ does not depend on $v$, then \eqref{Cond1:g_infty} is equivalent to 

$$\sup_{t\in [0,\infty)}\|g(t)\|^2_{L_2(S_0,H)}\leq \rho_{\infty}.$$

%Hence the condition  $\rho_{\infty}< \infty$ is equivalent to $\sup_{t\in [0,\infty)}\|g(t)\|^2_{L_2(S_0,H)}<\infty$.
\end{remark}

\section{The main result}\label{Section3}

We define the main quantity of interest in this paper, the \textit{bulk energy dissipation rate} (per unit mass)  for  (\ref{SSM}). This quantity will include dissipation due to the viscous forces and the turbulent diffusion.
For any function $\psi$ defined on $\mathbb{R}_+$ we denote by 
$$\left\langle\psi\right\rangle \coloneqq \limsup\limits_{T\rightarrow\infty}  \, \frac{1}{T} \int_{0}^{T} \psi(t)\, dt$$
the limsup of its time-average.  
\begin{definition}[Energy dissipation rate]\label{Def;Dissipation} 
The expected bulk energy dissipation rate (per unit mass)  for  (\ref{SSM}) is defined by
\begin{align}\label{varepsdef}
    \varepsilon:=  \varepsilon_0 + \varepsilon_M, 
\end{align}
where  
\begin{align*}
\varepsilon_0  \coloneqq 
\limsup\limits_{T\rightarrow\infty}\mathbb{E}\left[ \frac{1}{|D|}\,\frac{1}{T}\int_{0}^{T}   \nu\|\nabla u(t ,x, \omega)\|_{L^2(D)}^2  \, dt\right]\;=\frac{\nu}{|D|}\,\left\langle  \mathbb{E}[\|\nabla u\|_{L^2}^2]\right\rangle 
\end{align*}
and 
\begin{align*}
\varepsilon_M  \coloneqq 
\limsup\limits_{T\rightarrow\infty}\mathbb{E}\left[ \frac{1}{|D|}\,\frac{1}{T}\int_{0}^{T}   \bar{\nu}\|\nabla u(t ,x, \omega)\|_{L^r(D)}^r  \, dt\right]\; = \frac{\bar{\nu}}{|D|}\,\left\langle  \mathbb{E}[\|\nabla u\|_{L^r}^r ]\right\rangle.
\end{align*}

\end{definition}
\begin{remark}
Due to Fatou's lemma, the upper bound on the time-averaged dissipation rates $\varepsilon$, defined above,  does not necessarily yield a bound on the expected value of its time-limsup when the order of $\limsup$ and expectation are reversed.
\end{remark}

Before presenting the main result, it is essential to carefully define the various scales, taking into account both the domain and the underlying physics of the problem.

\begin{definition}\label{Scales} With the time-averaged notation $\left\langle \cdot \right \rangle$ introduced in Definition \ref{Def;Dissipation}, and 
with $|D| = \ell^3= $ the volume of the flow domain, the large-scale velocity \( U \) is defined as the square root of  the expectation of the $L^2(D)-$norm of velocity  

\[
U \coloneqq  \left \langle \frac{1}{|D|}  \mathbb{E}\left[ \|u\|^2  \right] \right \rangle ^{\frac{1}{2}}.
\]

The forcing length scale \( L \) is given by
$$L \coloneqq  \min \left\{\ell , \, \frac{F}{(\frac{1}{|D|} \|\grad f\|^2)^{\frac{1}{2}}}, \, \frac{F}{(\frac{1}{|D|} \|\grad f\|_{L^r}^r)^{\frac{1}{r}}}, \, \frac{F}{\|\grad f\|_{L^{\infty}(D)}}\right\},  \, $$
for non-zero forcing $f$, and by convention we take $L=\ell$ when  $f=0$.  Additionally, we introduce the quantities \( F \) and  \( G \)  to characterize the magnitudes of the deterministic and stochastic components of the force as

$$F \coloneqq  \left(\frac{1}{|D|  }\|f\|^2\right)^{\frac{1}{2}}=\frac{\|f\|}{|D|^{\frac{1}{2}}}, \hspace{1cm} \text{and}\hspace{1cm}
 G  \coloneqq   \left\langle \frac{1}{|D|  } \mathbb{E} \left[ \text{Tr} \left(g^* g (t, u)\right) \right]\right\rangle^{\frac{1}{2}}.$$

From here,  the Reynolds number is 
$$\Rey=\frac{U L}{\nu}$$
and the dimensionless  artificial  Reynolds number is given as
$$\rey=\frac{L^{r-1}}{\bar{\nu}\, U^{r-3}}. $$
We then define  the characteristic time $\tau$ as 
\[
  \tau= \frac{\rho_{\infty}\, L}{U}.
\]

\end{definition}

\begin{remark}
 It is straightforward to verify that the above expressions are dimensionally consistent. We interpret \( F \) as the amplitude of the deterministic forcing, and \( G^2 \) as the \textit{total energy rate supplied by the random force}, with physical units of \(\text{velocity}^2 \times \text{time}^{-1}\). Throughout, we assume \( F \, \text{and } G < \infty \).  From here and \eqref{Cond1:g_infty}, we verify that $\rho_{\infty}$ has physical dimension $\text{Time}^{-1}$; hence, $\tau$ is dimensionless. In addition, the applied deterministic  body force $f(x)$ is, without loss of generality, divergence free, and without loss of generality we restrict attention to mean-zero body forces and initial conditions so the velocity remains mean-zero for all $t > 0$.
\end{remark}

Recall that $\lambda_1$ is the smallest eigenvalue of  the Stokes operator (see \cite{FMRT01}).
We are now ready to give the rigorous statement of our main result.
  
\medskip
\begin{theorem}\label{MainThm1}

Let  $D = [0,\ell]^3$ denote the periodic box in $3d$, and $u(x, t; \omega)$ be a martingale  solution of the Stochastic Ladyzhynskaya  equations \eqref{SSM} on $[0,\infty)$ starting from the initial condition $u_0\in H$ that is $\mathcal{F}_0$-measurable. Suppose Assumption \ref{Assumption} holds with $\rho_{\infty} < {\nu\,  \lambda_1}$.   Then the averaged energy dissipation rate $\varepsilon$ in Definition \ref{Def;Dissipation} 
%\begin{align*}
%\varepsilon  \coloneqq 
%\limsup\limits_{T\rightarrow\infty}\mathbb{E}\left[ \frac{1}{|D|}\,\frac{1}{T}\int_{0}^{T}   \nu\|\nabla u(t ,\cdot, \omega)\|_{L^2}^2  + \bar{\nu}\|\nabla u(t ,\cdot, \omega)\|_{L^r}^r  \, dt\right]\;, 
%\end{align*}
satisfies
\begin{equation} \label{final}
\varepsilon \leq \, C  \,    \left(1+ \tau +  \frac{1}{\Rey}  +  \frac{1}{\rey}\right)\frac{U^3}{L}, 
\end{equation}
where $U$ is the mean value of the root-mean-square (space and time averaged), $L$ is the longest length scale in the applied forcing function, and  $\tau:=\frac{\rho_{\infty}\, L}{U}$ is a dimensionless number. 
%independent of $\nu$. 
%defined in Definition \ref{Scales}.

If, furthermore, $g(t,v)=g(t)$ does not depend on $v$, then we can omit the assumption  $\rho_{\infty} < {\nu\,  \lambda_1}$ and still obtain \eqref{final}.
\end{theorem}

\medskip

Before proving Theorem \ref{MainThm1}, we first  prove the boundedness of the kinetic energy under Assumption \ref{Assumption}. The argument for the model \eqref{SSM} closely follows the approach developed for the Navier--Stokes equations by Flandoli and Gatarek in \cite[Section 4]{FG95}, under the assumption that the magnitude of the noise can be controlled.

\begin{lemma}\label{KEBounded}
Suppose that Assumption \ref{Assumption} holds.
Then the kinetic energy of a  martingale solution to \eqref{SSM} is  uniformly bounded in time; i.e. 
%The uniform-in-time estimate (the kinetic energy of a  martingale solution to \eqref{SSM})
 \begin{equation}\label{UniformTimeBound}
  \sup_{t\in [0,\infty) } \mathbb{E} [\|u(t)\|^2]< \infty,
 \end{equation}
provided that $\rho_{\infty}< \nu \,\lambda_1$. If,   furthermore, $g(t,v)=g(t)$ in \eqref{SSM} does not depend on $v$, then \eqref{UniformTimeBound} holds provided that $\rho_{\infty}< \infty$.
%With $\rho_{T} \leq \frac{\nu, \lambda_1}{2}$,  the mean value of the kinetic energy of a  martingale solution to \eqref{SSM} on $[0, \infty)$ is uniformly bounded in time.
\end{lemma}

\begin{remark}
An assumption like $\rho_{\infty}< \nu \,\lambda_1$ is natural to guarantee a uniform-in-time estimate in the general case; such an assumption also appeared in  \cite{FG95}, for instance.
\end{remark}

\begin{proof}[Proof of Lemma \ref{KEBounded}]
    After integration by parts of \eqref{SSM}  with respect to $t$  and applying the It\^o formula (c.f. \cite[Theorem 4.32]{DZ14}) to \eqref{SSM2}, we obtain  the following energy inequality: for any $t\in [0,T]$, it holds $\mathbb{P}$-a.s.  that 
 \begin{equation}\label{EnergyEq2}
  \begin{split}
  \|u(t)\|^2 + 2  \int_0^t \nu \,  \|\nabla u(s)\|^2 ds  & + 2   \int_0^t \bar{\nu}\,  \|\nabla u(s)\|_{L^r}^r ds\,   \leq  \|u_0\|^2   + \int_0^t \text{Tr} (g^* g (s, u(s))) ds  \\
  & +  2 \int_0^t  \langle  f, u(s)\rangle ds + 2 \int_0^t  \big\langle g(s, u(s)),\, u(s)\big\rangle  \, dW(s).
  \end{split}
  \end{equation}

%{\color{red}Note that\[|\overline{u(s)}|=|(1,u(s))|\leq \int |u(s)| \leq \|u(s)\|_{L^2}\]}
 
Using the \ref{Young}, one can estimate the deterministic source term above as follows: for any $\epsilon\in(0,\infty)$, 
\begin{eqnarray}\left|\int_0^t  \langle  f, u(s)\rangle ds \right| 
\leq& \int_0^t \frac{1}{2\epsilon}\|f\|^2+\frac{\epsilon}{2}\|u(s)\|^2 \,ds.
\end{eqnarray}
Applying the \ref{Poincare} $\lambda_1 \|u(s)\|^2 \leq  \|\nabla u(s)\|^2$ to the second term and taking $\epsilon=\nu\lambda_1$, one obtains
$$\left|\int_0^t  \langle f, u(s) \rangle ds \right| \leq \frac{1}{2\,\nu\, \lambda_1} \int_0^t \|f\|^2 ds + \frac{\nu}{2} \int_0^t  \|\nabla u\|^2 ds.  $$
%{\color{red}(Louis: Sorry, I don't understand why we have the above inequality. Using \ref{Young}, we obtain $fu\leq \frac{1}{2\nu}f^2 +\frac{\nu}{2}u^2$. Using this and then  the Poincar\'e inequality $\lambda_1 \|u(s)\|^2 \leq  \|\nabla u(s)\|^2$, we obtain
%\begin{equation*}  \begin{split}\left|\int_0^t  \langle  f, u(s)\rangle ds \right| \leq& \frac{1}{2\,\nu} \int_0^t \|f\|^2 ds + \frac{\nu}{2} \int_0^t  \| u(s)\|^2 ds \\\leq& \frac{1}{2\,\nu} \int_0^t \|f\|^2 ds + \frac{\nu}{2\lambda_1} \int_0^t  \| \nabla u(s)\|^2 ds  \end{split}  \end{equation*})}
%{\color{blue}( Ali: You could also stay in $L^2$, but you need  to consider the Poincare coefficient $1/{\lambda}^{1/2}$ attached to $f$ before applying the Young as you need $ \frac{\nu}{2}  \| u(s)\|^2 $ on the rhs to absorb it into the viscosdity term on the lhs . It is modified now.)}
With the above estimate,
and applying  the \ref{Poincare} again, 
we obtain from \eqref{EnergyEq2} the following:
 \begin{equation}\label{EnergyEq2b}
  \begin{split}
  \|u(t)\|^2 + &  \int_0^t \nu\,\lambda_1 \,  \| u(s)\|^2 ds   + 2   \int_0^t \bar{\nu}\,  \|\nabla u(s)\|_{L^r}^r ds\,   \leq  \|u_0\|^2   + \int_0^t \text{Tr} (g^* g (s, u(s))) ds  \\
  & +  \frac{1}{\nu\, \lambda_1} \int_0^t \|f\|^2 ds  + 2 \int_0^t  \big\langle g(s, u(s)),\, u(s)\big\rangle  \, dW(s)
  \end{split}
  \end{equation}
$\mathbb{P}$-a.s., for $t\in[0,T]$.

Next,  with Assumption  \ref{Assumption},  we have that
$\text{Tr} (g^* g (s, u(s)))\leq \rho_{T}(1+\|u(s)\|^2)$.  Applying this,
we obtain from \eqref{EnergyEq2b} the following after taking expectation and rearranging terms involving $\E[\|u(s)\|^2]$: for all $t\in[0,T]$,
\begin{equation}\label{EnergyEq2c}
  \begin{split}
  \E[\|u(t)\|^2] +   \int_0^t (\nu \,  \lambda_1-\rho_{T}) \,\E[\|u(s)\|^2]\, ds   \,   \leq  \E[\|u_0\|^2]  &  + \rho_{T}\,t  
   +  \frac{1}{\nu \, \lambda_1} \int_0^t \|f\|^2  ds.
  \end{split}
  \end{equation}

Since $f = f(x)$ is time-independent and the map $t \mapsto \mathbb{E}[\|u(t)\|^2]$ is absolutely continuous, the above inequality can be written as
\begin{equation*}
\frac{d}{dt}\mathbb{E}[\|u(t)\|^2]
+ (\nu \lambda_1 - \rho_T)\,\mathbb{E}[\|u(t)\|^2]
\leq
\rho_T + \frac{1}{\nu \lambda_1}\|f\|^2 .
\end{equation*}
Provided that $ \sup_{T \in (0,\infty)} \rho_T < \nu \lambda_1$, Assumption~\ref{Assumption}, we obtain
\begin{align}
\frac{d}{dt}\mathbb{E}[\|u(t)\|^2]
+ \alpha \,\mathbb{E}[\|u(t)\|^2]
\leq
\nu \lambda_1 + \frac{1}{\nu \lambda_1}\|f\|^2,
\end{align}
for some $\alpha > 0$. Setting
\[
Y(t) = \mathbb{E}[\|u(t)\|^2],
\qquad
C = \nu \lambda_1 + \frac{1}{\nu \lambda_1}\|f\|^2,
\]
the inequality becomes
\[
Y'(t) + \alpha Y(t) \leq C.
\]
Applying the integrating factor $e^{\alpha t}$ to this differential inequality yields the uniform-in-time estimate \eqref{UniformTimeBound}:
\begin{equation}
\sup_{t \geq 0} \mathbb{E}\big[\|u(t)\|^2\big]
\leq
\mathbb{E}\big[\|u_0\|^2\big]
+ \frac{\nu \lambda_1}{\alpha}
+ \frac{1}{\alpha \nu \lambda_1}\|f\|^2
< \infty.
\end{equation}  

In the special case where $g(t,v)=g(t)$ is independent of $v$, the estimate \eqref{EnergyEq2c} can be strengthened. In this case, we obtain
\begin{equation}\label{EnergyEq2d}
\begin{split}
\E[\|u(t)\|^2]
+ \int_0^t \nu \lambda_1 \,\E[\|u(s)\|^2]\, ds
\leq
\E[\|u_0\|^2]
+ \rho_T \, t
+ \frac{1}{\nu \lambda_1} \int_0^t \|f\|^2 \, ds.
\end{split}
\end{equation}
This estimate likewise yields uniform-in-time bounds for $\E[\|u(t)\|^2]$.

\end{proof}
We are now ready to prove the main result.
\begin{proof}[Proof of Theorem \ref{MainThm1}]
    The proof starts from the energy type of integral inequality \eqref{EnergyEq2}, which  can be formally short-handed  as the following stochastic differential equation
 \begin{equation}\label{EnergyIneq1}
  \begin{split}
 d\|u(t)\|^2 + 2 \nu \|\nabla u (t) \|^2 dt +  2   \bar{\nu}\,  \|\nabla u(t)\|_{L^r}^r dt\ &\leq \text{Tr} (g^* g (t, u)) dt + 2 \langle f(x), u(t)\rangle dt \\
 & + 2 \,   \langle g(t, u), u \rangle \, dW(t).
  \end{split}
  \end{equation}
Using Hölder's inequality, one can show that $(g(t,u), u) \in L^2(0,T)$, which implies
\[
\mathbb{E} \left[ \int_0^T \langle g(t,u), u \rangle \, dW \right] = 0.
\]
Therefore, by averaging equation \eqref{EnergyEq2} over the interval $[0, T]$, applying the Cauchy–Schwarz inequality, and taking the expectation $\E$ with respect to $\mathbb{P}$, we arrive at
\begin{equation*}\label{Eq-3}
\begin{split}
&  \frac{1}{ 2\, |D|\, T } \, \mathbb{E} \left[\|u(T)\|^2 - \|u(0)\|^2 \right] + \mathbb{E} \big[\frac{1}{|D|} \frac{1}{T}  \int_0^T \nu \|\nabla u(t)\|^2 + \bar{\nu } \|\nabla u(t)\|_{L^r}^r dt \big]  \\  
& \leq \,  \frac{1}{2}  \frac{1}{|D|} \frac{1}{T} \int_0^T \mathbb{E} \left[ \text{Tr} \left(g^* g (t, u)\right) \right] dt + (\frac{1}{T} \frac{1}{|D|} \int_0^T \|f\|^2 dt)^{\frac{1}{2}} \, \E \big[(\frac{1}{T} \frac{1}{|D|} \int_0^T \|u(t)\|^2 dt)^{\frac{1}{2}}\big]\\
& \leq  \,  \frac{1}{2}  \frac{1}{|D|} \frac{1}{T} \int_0^T \mathbb{E} \left[ \text{Tr} \left(g^* g (t, u)\right) \right] dt + (\frac{1}{T} \frac{1}{|D|} \int_0^T \|f\|^2 dt)^{\frac{1}{2}} \,  \left( \frac{1}{T} \frac{1}{|D|} \E \left[ \int_0^T \|u(t)\|^2 dt \right]\right)^{\frac{1}{2}}.
  \end{split}
\end{equation*}
where the last step follows by Jensen's inequality.  After taking the $\limsup$ of the above inequality  and using the continuity and monotonicity of the square root function, one can apply the scaling defined in Definition \ref{Scales} to obtain
\begin{equation}\label{FirstIneq_combined}
\varepsilon \leq \frac{1}{2} \,  G^2  + F \, U.
\end{equation}

The next step is to estimate the magnitudes of the determintic force  $F$ in order to obtain a bound for \eqref{FirstIneq_combined}. We begin by taking the inner product of \eqref{SSM} with $f(x)$ and applying integration by parts, which yields

\begin{equation*}
\begin{split}
(du, f) + (u\cdot \nabla u, f)dt  & + \nu\, ( \nabla u,\nabla f) dt + \bar{\nu }\, (|\nabla u|^{r-2} \nabla u, \nabla f ) \, dt \\
& = \|f\|^2 dt  +   (g(t, u), f)\,  dW(t).
\end{split}
\end{equation*}
Integrating the above equation  in time over $[0,T]$, dividing by $|D|\,  T$, and taking the expectation with respect to $\mathbb{P}$, we obtain

\begin{equation}\label{Eq1}
\begin{split}
 &  \frac{\|f\|^2}{|D|}  +    
 \frac{1}{|D|}\, \frac{1}{T} \E \left[\int_0^T (g(t, u), f)\,  dW(t)\right] \\
 & =  \frac{1}{|D|}\, \frac{1}{T} \E \left[\, \int_0^T (du, f) + (u\cdot \nabla u, f)dt  + \nu\, ( \nabla u,\nabla f) dt + \bar{\nu }\, (|\nabla u|^{r-2} \nabla u, \nabla f ) \, dt  \, \right]\\
 & = \RN{1}+ \RN{2}+ \RN{3}+\RN{4}.
  \end{split}
\end{equation}

By our assumption on $u$ and $g$,
$\E \left[\int_0^T (g(t, u), f)\,  dW(t)\right]= 0$ in \eqref{Eq1}.
Before passing to the $\limsup$, we focus on estimating the four terms $\RN{1}, \RN{2}, \RN{3}, \RN{4}$ on the right-hand side of equation \eqref{Eq1}.

Using boundedness of the kinetic energy in Lemma \ref{KEBounded}, and having $f(x)$ being time-independent, term $\RN{1}$ will be estimated as 
\begin{align}\label{Term1}
  \E\left[\frac{1}{ T}\int_0^T (du,f) \right]&=\E\left[\frac{\left(u(T),f\right)-\left(u(0),f\right)}{ T}\right] \sim \mathcal{O}(\frac{1}{T}).
  \end{align}

On Term $\RN{2}$,  since $\nabla \cdot u = 0 $ we have $ (u\cdot \nabla u, f) = (\nabla \cdot (u\otimes u),   f)$ , and using integration  by parts we obtain,
\begin{align}\label{Term2}
    \left|\frac{1}{|D|}\frac{1}{T} \int_0^T \E \left[(u\cdot \nabla u, f)\right] dt \right| &\leq  \left| \frac{1}{|D|}\frac{1}{T} \int_0^T \E \left[ (u \otimes u, \nabla f) \right] dt  \right|\notag\\
    & \leq \|\nabla f\|_{L^{\infty}(0,T; L^{\infty}(D))} \frac{1}{|D|}\frac{1}{T} \int_0^T  \E \left[ \|u\|^2  \right] dt.  
\end{align}

To estimate the term $\RN{3}$  in (\ref{Eq1}), by using the Cauchy-Schwarz-Young inequality and Definition  \ref{Scales} we have, 
\begin{equation}\label{Term3}
\begin{split}
\left | \E \left[ \frac{1}{|D|}\frac{1}{T}\int_0^T \nu ( \nabla u,\nabla f) dt \right] \right| &\leq  \E \left[(\frac{1}{|D|}\frac{1}{T} \int_0^T \nu \|\nabla u\|^2 dt)^{\frac{1}{2}} \right] \, \left(\frac{1}{|D|}\frac{1}{T} \int_0^T \nu \|\nabla f\|^2 dt\right)^{\frac{1}{2}}\\
&\leq  \left( \E \left[\frac{1}{|D|}\frac{1}{T} \int_0^T \nu \|\nabla u\|^2 dt \right] \right)^{\frac{1}{2}} \, \nu^{\frac{1}{2}}\left(\frac{1}{|D|}\frac{1}{T} \int_0^T  \|\nabla f\|^2 dt\right)^{\frac{1}{2}},
\end{split}
\end{equation}
where the last step is justified by Jensen's inequality.

And finally using \ref{Holder} twice  with $p=\frac{r}{r-1}$ and $p'=r$ on term $\RN{4}$ in \eqref{Eq1},   along with Jensen's inequality, we get to 
\begin{equation}\label{Term4}
\begin{split}
&\left | \E \left[\frac{1}{|D|}\frac{1}{T}\int_0^T \bar{\nu}\,  (|\nabla u|^{r-2} \nabla u, \nabla f ) \, dt \right] \right| \leq \frac{1}{|D|}\frac{1}{T} \E \left[\int_0^T  \bar{\nu}\,   \|\nabla u\|_{L^r}^{r-1} \|\nabla f\|_{L^r}\, dt \right]\\
&\leq  \E \left[ \left( \frac{1}{|D|}\frac{1}{T}\int_0^T  \bar{\nu}\,   \|\nabla u\|_{L^r}^{r} \, dt\right)^{\frac{r-1}{r}} \right] \, \left(\bar{\nu}\,  \frac{1}{|D|}\frac{1}{T}\int_0^T   \|\nabla f\|_{L^r}^r\, dt  \right)^{\frac{1}{r}}\\
& \leq     \left(\mathbb{E}\left[   \frac{1}{|D|}\frac{1}{T}\int_0^T  \bar{\nu}\,   \|\nabla u\|_{L^r}^{r} \, dt\right] \right)^{\frac{r-1}{r}} \,  \bar{\nu}^{\frac{1}{r}}\,\left(  \frac{1}{|D|}\frac{1}{T}\int_0^T   \|\nabla f\|_{L^r}^r\, dt  \right)^{\frac{1}{r}}. 
\end{split}
\end{equation}

Now, considering the statistical quantities in Definition \ref{Def;Dissipation}, and  scales $L, U,$ and $F$ given in Definition~\ref{Scales}, we apply $\limsup$ to each of the terms on the right-hand side of Equation~\eqref{Eq1}. Using the estimates obtained above in \eqref{Term1}, \eqref{Term2}, \eqref{Term3}, and \eqref{Term4}, we arrive at
\begin{equation}\label{Eq4}
\begin{split}
    \limsup_{T \to \infty}\frac{1}{|\Omega|}\, \frac{1}{T}\,  \E \left[\, \int_0^T (du, f)  \, \right] & \leq \mathcal{O}(\frac{1}{T}) \to 0,\\
    \limsup_{T \to \infty}  \frac{1}{|\Omega|}\, \frac{1}{T} \, \E \left[\,  \int_0^T (u\cdot \nabla u, f)\, dt  \,  \right] & \leq \frac{F}{L}\, U^2,\\
      \limsup_{T \to \infty} \frac{1}{|\Omega|}\, \frac{1}{T} \, \E \left[\,  \nu\, \int_0^T ( \nabla u,\nabla f)\,  dt  \, \right] & \leq \sqrt{\nu}\,  \frac{F}{L} \, \varepsilon_0 ^{\frac{1}{2}},\\
       \limsup_{T \to \infty}  \frac{1}{|\Omega|}\, \frac{1}{T} \, \E \left[\,\bar{\nu }\, \int_0^T (|\nabla u|^{r-2} \nabla u, \nabla f ) \, dt  \, \right] & \leq  \bar{\nu}^{\frac{1}{r}}\,  \frac{F}{L}\,  \varepsilon_M ^{\frac{r-1}{r}}.
\end{split}
\end{equation}
%{\color{red}Did we missed the integral $\int_0^T$ in the last 3 inequalities of \eqref{Eq4}?}

It is important to note that the ability to pass the $\limsup$ through the root functions in the argument above hinges on the fact that the root functions are both continuous and monotone increasing.

After taking $\limsup$ from \eqref{Eq1}, and inserting the estimates in \eqref{Eq4}, we have

\begin{equation*}
 F \leq \frac{U^2}{L} + \frac{\sqrt{\nu}}{L}  \, \varepsilon_0 ^{\frac{1}{2}} +    \frac{\bar{\nu}^{\frac{1}{r}}}{L}\,  \varepsilon_M^{\frac{r-1}{r}}.
\end{equation*}
After inserting multipliers of $U^{\frac{1}{2}}$ and $U^{-\frac{1}{2}}$ in the second term above and also $U^{\frac{r-1}{r}}$ and  $U^{\frac{1-r}{r}}$ in the third  term, we apply the Young's inequality to arrive at 
\begin{equation}\label{SecondIneq1}
F \leq \frac{U^2}{L} +  \frac{1}{2} \frac{U \, \nu}{L^2}+\frac{1}{2} \frac{\varepsilon_0  }{U} + \frac{1}{r} \frac{\nu}{L^r}\, U^{r-1} + \frac{r-1}{r} \frac{\varepsilon_M }{U}\, .
\end{equation}
Using the  above estimate  (\ref{SecondIneq1}) for $F$ in (\ref{FirstIneq_combined}) gives
\begin{equation*}
 \varepsilon_0 +  \varepsilon_M  =  \varepsilon  \leq \frac{1}{2}  G^2  + \frac{U^3}{L} +  \frac{1}{2} \frac{U^2 \, \nu}{L^2}+\frac{1}{2} \varepsilon_0   + \frac{1}{r} \frac{\nu}{L^r}\, U^{r} + \frac{r-1}{r}\varepsilon_M. 
\end{equation*}
We have thus estimated the  upper bound as, 
\begin{equation} \label{Semi-final}
\varepsilon \lesssim \,    G^2 + \left(1+ \frac{\nu}{U\, L} +  \frac{\bar{\nu}\, U^{r-3}}{L^{r-1}} \right)\frac{U^3}{L}.
\end{equation}

Now we focus on estimating $G^2$ in the above. 
Using  Assumption \ref{Assumption}, the stochastic forcing term $G^2$ in the above can be estimated as
\begin{equation}\label{G^2}
    \begin{split}
        G^2  & \coloneqq   \left\langle \frac{1}{|D|  } \mathbb{E} \left[ \text{Tr} \left(g^* g (t, u(t))\right) \right]\right\rangle   = \left\langle \frac{1}{|D|  } \mathbb{E} \left[ \| g (t,u(t))\|^2\right] \right\rangle\\
        &\leq  \rho_{\infty}\, \left\langle \frac{1}{|D|  } \mathbb{E} \left[(1+|u(t)|_H^2)\right]\right\rangle \leq C\,\rho_{\infty}\, U^2,
    \end{split}
\end{equation}
where $C$ is a dimensionless constant. Using the above estimate for \( G^2 \) in \eqref{Semi-final}, and the scales defined in Definition \ref{Scales},  we arrive at \eqref{final} as claimed.

%Note  $\tau$ depends on $U$ and features of the  forcing through $L$ and $\rho_{\infty}$, but not on $\nu$.

\end{proof}

\section{Summary $\&$  Discussion}

We study a stochastic Ladyzhenskaya--Smagorinsky equation on a three-dimensional periodic domain, where the flow is driven by a combination of deterministic and stochastic forcing of the form \( f\,dt + g(t,u)\,dW \).

A central question in the theory of turbulence is whether the energy dissipation rate remains strictly positive in the high Reynolds number (vanishing viscosity) limit; the phenomenon is commonly referred to as anomalous dissipation. The estimate established in Theorem~\ref{MainThm1} shows that the first moment of the energy dissipation rate satisfies $ \varepsilon \lesssim  C  \frac{U^3}{L}$, 
which is dimensionally consistent and scales with the classical energy input rate $U^3/L$.  In particular, the bound remains finite as $\nu \to 0$, demonstrating that the model is consistent with the dissipative anomaly predicted by Kolmogorov’s theory. Moreover, the contribution of the stochastic forcing enters only through the dimensionless parameter $\rho_{\infty} L/U$, preserving the expected scaling structure. In the deterministic limit ($g\sim 0$), the estimate reduces to the known dissipation rates for the Navier--Stokes equations and the Smagorinsky model, confirming consistency with established results.  Thus, while our analysis does not establish a positive lower bound in the inviscid limit, it rigorously shows that the stochastic turbulence model exhibits the correct scaling behavior associated with anomalous dissipation and does not over-dissipate in the absence of boundary layer effects.

%Motivated by the statistical nature of turbulence and the phenomenon of anomalous dissipation, we investigated the time-averaged energy dissipation rate in the presence of space-time Gaussian noise.  We derive an upper bound on the first moment of the energy dissipation rate, showing that it remains finite in the vanishing viscosity limit and scales consistently with Kolmogorov’s theory of turbulence. 

Throughout this paper, the deterministic force $f = f(x)$ is assumed to be time-independent. 
In the case of a time-dependent deterministic force, additional care would be required, both in the stochastic analysis and in the proper quantification of the forcing amplitude, which would naturally involve a suitable time-averaged formulation.  We do not pursue this direction here; however, extending the present results to a broader class of time-dependent deterministic forces would be an interesting topic for future research (see, e.g., \cite{kumar2025}). 

Our result extends previous estimates for both the deterministic Navier–Stokes  equations \cite{DF02} and Smagorinsky models \cite{L16} and demonstrates that the stochastic system \eqref{SSM}, in the absence of the boundary layer,   preserves physically relevant dissipation scaling without artificial over-dissipation.  The extension of this work to the channel flow  remains an open problem, which would provide valuable insights into near-wall behavior.  

Having a uniform-in-time bound on the kinetic energy $ \E \left[\|u\|^2\right]$ is a key ingredient in the proof. For general multiplicative noise of the form \( g = g(t,u) \) in \eqref{SSM}, such a bound  established in Lemma \ref{KEBounded} under a smallness assumption on the diffusion term (as also considered in \cite{FG95}). However, the results in this manuscript can be extended to the case of additive noise, i.e., \( g = g(t) \), without requiring any additional assumptions.

Our main analysis began with the energy inequality \eqref{EnergyEq2}, then  in  \eqref{FirstIneq_combined},  we first bound the relevant statistical quantities in terms of the amplitude of the deterministic force \( F \) and the total energy rate supplied by the random force \( G^2 \). We then quantify \( F \) in terms of the characteristic velocity \( U \)  and length scales \( L \),  and the dissipation quantities by taking the inner product of \( f \) with the equation of motion, see \eqref{SecondIneq1}. One may be tempted to follow the same strategy to  estimate \( G^2 \) by taking the inner product of \( g(t,u) \) with the equation \eqref{SSM}. However, after taking the expectation, the quantity involving  $G^2$ vanishes and 
it is not clear how to estimate $G^2$.
%nothing remains to be estimated directly; as $\E \left[\int_0^T g^2(t,u) \, dW\right]=0$. 
To overcome this, we rely on the standard  Assumption \ref{Assumption} to quantify \( G^2 \). A more refined estimate of this quantity in terms of the characteristic scales \( U \) and \( L \) would further strengthen our results and provide deeper insight into the balance between the stochastic  force and dissipation.

\section*{Acknowledgments}

This work of L. W.-T. Fan    was  supported by the
NSF grants DMS-2534011 and DMS-2532574.  This work of A. Pakzad was  supported by the
NSF grant DMS-2532987.


\begin{thebibliography}{99}

\bibitem{alexakisdoering2006}
A. Alexakis and C. R. Doering,
{\it Energy and enstrophy dissipation in steady state 2D turbulence},
Physica D {\bf 223} (2006), 82--87.

\bibitem{baumert2005}
H. Z. Baumert, J. H. Simpson and J. S\"undermann (eds.),
{\it Marine Turbulence: Theories, Observations, and Models},
Cambridge Univ. Press, Cambridge, 2005.

\bibitem{B01}
V. Barbu,
{\it Stabilization of Navier--Stokes Flows},
Springer, London, 2001.

\bibitem{BCPW19}
J. Bedrossian, M. Coti Zelati, S. Punshon-Smith and F. Weber,
{\it A sufficient condition for the Kolmogorov 4/5 law for stationary martingale solutions to the 3D Navier--Stokes equations},
Comm. Math. Phys. {\bf 367} (2019), 1045--1075.

\bibitem{BT73}
A. Bensoussan and R. Temam,
{\it \'Equations stochastiques du type Navier--Stokes},
J. Funct. Anal. {\bf 13} (1973), 195--222.

\bibitem{biferale2004anomalous}
L. Biferale, M. Cencini, A. S. Lanotte, M. Sbragaglia and F. Toschi,
{\it Anomalous scaling and universality in hydrodynamic systems with power-law forcing},
New J. Phys. {\bf 6} (2004), 37.

\bibitem{biferale2004}
L. Biferale, A. S. Lanotte and F. Toschi,
{\it Effects of forcing in three-dimensional turbulent flows},
Phys. Rev. Lett. {\bf 92} (2004), 094503.

\bibitem{BJMT14}
A. Biswas, M. S. Jolly, V. R. Martinez and E. S. Titi,
{\it Dissipation length scale estimates for turbulent flows: A Wiener algebra approach},
J. Nonlinear Sci. {\bf 24} (2014), 441--471.

\bibitem{B00}
H. Breckner,
{\it Galerkin approximation and the strong solution of the Navier--Stokes equation},
J. Appl. Math. Stochastic Anal. {\bf 13} (2000), 239--259.

\bibitem{BP00}
Z. Brze\'zniak and S. Peszat,
{\it Infinite Dimensional Stochastic Analysis},
R. Neth. Acad. Arts Sci., Amsterdam, 2000.

\bibitem{B70}
F. H. Busse,
{\it Bounds for turbulent shear flow},
J. Fluid Mech. {\bf 41} (1970), 219--240.

\bibitem{CLM01}
T. Caraballo, K. Liu and X. R. Mao,
{\it On stabilization of partial differential equations by noise},
Nagoya Math. J. {\bf 161} (2001), 155--170.

\bibitem{CI08}
P. Constantin and G. Iyer,
{\it A stochastic Lagrangian representation of the three-dimensional incompressible Navier--Stokes equations},
Comm. Pure Appl. Math. {\bf 61} (2008), 330--345.

\bibitem{CI11}
P. Constantin and G. Iyer,
{\it A stochastic-Lagrangian approach to the Navier--Stokes equations in domains with boundary},
Ann. Appl. Probab. {\bf 21} (2011), 1466--1492.

\bibitem{cyr2020review}
J. Cyr, P. Nguyen, S. Tang and R. Temam,
{\it Review of local and global existence results for stochastic PDEs with L\'evy noise},
Discrete Contin. Dyn. Syst. {\bf 40} (2020), 5639--5710.

\bibitem{DZ14}
G. Da Prato and J. Zabczyk,
{\it Stochastic Equations in Infinite Dimensions},
Cambridge Univ. Press, Cambridge, 2014.

\bibitem{DGTZ12}
A. Debussche, N. Glatt-Holtz, R. Temam and M. Ziane,
{\it Global existence and regularity for the 3D stochastic primitive equations of the ocean and atmosphere with multiplicative white noise},
Nonlinearity {\bf 25} (2012), 2093--2118.

\bibitem{DLPRSZ18}
V. DeCaria, W. Layton, A. Pakzad, Y. Rong, N. Sahin and H. Zhao,
{\it On the determination of the grad-div criterion},
J. Math. Anal. Appl. {\bf 467} (2018), 1032--1037.

\bibitem{Dudley2024}
E. Dudley,
{\it Necessary and sufficient conditions for Kolmogorov's flux laws on $\mathbb T^2$ and $\mathbb T^3$},
Nonlinearity {\bf 37} (2024).

\bibitem{DC92}
C. R. Doering and P. Constantin,
{\it Energy dissipation in shear driven turbulence},
Phys. Rev. Lett. {\bf 69} (1992), 1648--1651.

\bibitem{DF02}
C. R. Doering and C. Foias,
{\it Energy dissipation in body-forced turbulence},
J. Fluid Mech. {\bf 467} (2002), 289--306.

\bibitem{Evans2013}
L. C. Evans,
{\it An Introduction to Stochastic Differential Equations},
Amer. Math. Soc., Providence, RI, 2013.

\bibitem{FJP21}
W. L. Fan, M. Jolly and A. Pakzad,
{\it Three-dimensional shear driven turbulence with noise at the boundary},
Nonlinearity {\bf 34} (2021), 4764--4786.

\bibitem{FPTT23}
W. L. Fan, A. Pakzad, K. Tawri and R. Temam,
{\it 3D shear flows driven by L\'evy noise at the boundary},
Probab. Uncertain. Quant. Risk {\bf 8} (2023).

\bibitem{FSQD19}
K. Fellner, S. Sonner, B. Q. Tang and D. D. Thuan,
{\it Stabilisation by noise on the boundary for a Chafee--Infante equation with dynamical boundary conditions},
Discrete Contin. Dyn. Syst. {\bf 24} (2019), 4055--4078.

\bibitem{FG95}
F. Flandoli and D. G\polhk atarek,
{\it Martingale and stationary solutions for stochastic Navier--Stokes equations},
Probab. Theory Related Fields {\bf 102} (1995), 367--391.

\bibitem{FGHR08}
F. Flandoli, M. Gubinelli, M. Hairer and M. Romito,
{\it Rigorous remarks about scaling laws in turbulent fluids},
Comm. Math. Phys. {\bf 278} (2008), 1--29.

\bibitem{FMRT01}
C. Foias, O. Manley, R. Rosa and R. Temam,
{\it Navier--Stokes Equations and Turbulence},
Cambridge Univ. Press, Cambridge, 2001.

\bibitem{NTT21}
P. Nguyen, K. Tawri and R. Temam,
{\it Nonlinear stochastic parabolic partial differential equations with a monotone operator of the Ladyzhenskaya--Smagorinsky type driven by L\'evy noise},
J. Funct. Anal. {\bf 281} (2021), 109157.

\bibitem{F95}
U. Frisch,
{\it Turbulence: The Legacy of A. N. Kolmogorov},
Cambridge Univ. Press, Cambridge, 1995.

\bibitem{GKVZ14}
N. Glatt-Holtz, I. Kukavica, V. Vicol and M. Ziane,
{\it Existence and regularity of invariant measures for the three-dimensional stochastic primitive equations},
J. Math. Phys. {\bf 55} (2014), 051504.

\bibitem{H72}
L. N. Howard,
{\it Bounds on flow quantities},
Annu. Rev. Fluid Mech. {\bf 4} (1972), 473--494.

\bibitem{K16}
R. R. Kerswell,
{\it Energy dissipation rate limits for flow through rough channels and tidal flow across topography},
J. Fluid Mech. {\bf 808} (2016), 562--575.

\bibitem{K41}
A. N. Kolmogorov,
{\it The local structure of turbulence in incompressible viscous fluid for very large Reynolds numbers},
Proc. Roy. Soc. London Ser. A {\bf 434} (1991), 9--13.

\bibitem{KUZ18}
I. Kukavica, K. U\u{g}urlu and M. Ziane,
{\it On the Galerkin approximation and strong norm bounds for the stochastic Navier--Stokes equations with multiplicative noise},
Differential Integral Equations {\bf 31} (2018), 173--186.

\bibitem{kumar2025}
A. Kumar and A. Pakzad,
{\it Statistical estimates for 2D stochastic Navier--Stokes equations},
J. Stat. Phys. {\bf 192} (2025).

\bibitem{K99}
A. A. Kwieci\'nska,
{\it Stabilization of partial differential equations by noise},
Stochastic Process. Appl. {\bf 79} (1999), 179--184.

\bibitem{L67}
O. A. Ladyzhenskaya,
{\it New equations for the description of the motions of viscous incompressible fluids and global solvability for their boundary value problems},
Trudy Mat. Inst. Steklov. {\bf 102} (1967), 85--104.

\bibitem{L68}
O. A. Ladyzhenskaya,
{\it On modifications of Navier--Stokes equations for large gradients of velocities},
Zap. Nauchn. Sem. LOMI {\bf 7} (1968), 126--154.

\bibitem{L69}
O. A. Ladyzhenskaya,
{\it The Mathematical Theory of Viscous Incompressible Flow},
Gordon and Breach, New York, 1969.

\bibitem{L98}
O. A. Ladyzhenskaya,
{\it Some results on modifications of three-dimensional Navier--Stokes equations},
in {\it Nonlinear Analysis and Continuum Mechanics}, Springer, New York, 1998, 73--84.

\bibitem{L16}
W. J. Layton,
{\it Energy dissipation in the Smagorinsky model of turbulence},
Appl. Math. Lett. {\bf 59} (2016), 56--59.

\bibitem{MS02}
J. L. Menaldi and S. S. Sritharan,
{\it Stochastic 2D Navier--Stokes equation},
Appl. Math. Optim. {\bf 46} (2002), 31--30.

\bibitem{MR04}
R. Mikulevicius and B. L. Rozovskii,
{\it Stochastic Navier--Stokes equations for turbulent flows},
SIAM J. Math. Anal. {\bf 35} (2004), 1250--1310.

\bibitem{MR05}
R. Mikulevicius and B. L. Rozovskii,
{\it Global $L^2$-solutions of stochastic Navier--Stokes equations},
Ann. Probab. {\bf 33} (2005), 137--176.

\bibitem{AP17}
A. Pakzad,
{\it Damping functions correct over-dissipation of the Smagorinsky model},
Math. Methods Appl. Sci. {\bf 40} (2017), 5933--5945.

\bibitem{AP19}
A. Pakzad,
{\it On the long time behavior of time relaxation model of fluids},
Physica D {\bf 408} (2020).

\bibitem{PS2024}
A. Pakzad and F. Siddiqua,
{\it Statistics in a backscatter eddy viscosity turbulence model},
Numer. Methods Partial Differential Equations {\bf 41} (2025).

\bibitem{P00}
S. B. Pope,
{\it Turbulent Flows},
Cambridge Univ. Press, Cambridge, 2000.

\bibitem{richardson1922}
L. F. Richardson,
{\it Weather Prediction by Numerical Process},
Cambridge Univ. Press, Cambridge, 1922.

\bibitem{RocknerZhuZhu2015}
M. R\"ockner, R. Zhu and X. Zhu,
{\it Sub and supercritical stochastic quasi-geostrophic equation},
Ann. Probab. {\bf 43} (2015), 1202--1273.

\bibitem{SM63}
J. Smagorinsky,
{\it General circulation experiments with the primitive equations. I. The basic experiment},
Mon. Weather Rev. {\bf 91} (1963), 99--164.

\bibitem{sreenivasan1984}
K. R. Sreenivasan,
{\it On the scaling of the turbulence energy dissipation rate},
Phys. Fluids {\bf 27} (1984), 1048--1051.

\bibitem{sreenivasan1998}
K. R. Sreenivasan,
{\it An update on the energy dissipation rate in isotropic turbulence},
Phys. Fluids {\bf 10} (1998), 528--529.

\bibitem{VR50}
J. von Neumann and R. D. Richtmyer,
{\it A method for the numerical calculation of hydrodynamic shocks},
J. Appl. Phys. {\bf 21} (1950), 232--237.

\bibitem{WW15}
D. Wang and H. Wang,
{\it Global existence of martingale solutions to the three-dimensional stochastic compressible Navier--Stokes equations},
Differential Integral Equations {\bf 28} (2015), 1105--1154.

\bibitem{W00}
X. Wang,
{\it Effect of tangential derivative in the boundary layer on time averaged energy dissipation rate},
Physica D {\bf 144} (2000), 142--153.

\bibitem{W97}
X. Wang,
{\it Time-averaged energy dissipation rate for shear driven flows in $\mathbb{R}^n$},
Physica D {\bf 99} (1997), 555--563.

\bibitem{Wilcox06}
D. C. Wilcox,
{\it Turbulence Modeling for CFD},
3rd ed., DCW Industries, La Ca\~nada, CA, 2006.

\end{thebibliography}
\end{document}